\documentclass[11pt,letterpaper]{amsart}
\usepackage{mathtools}
\usepackage{mathrsfs}
\usepackage[OT2, T1]{fontenc}
\usepackage{url}
\usepackage{amsmath}
\usepackage{array}
\usepackage{graphicx}
\usepackage{amsfonts}
\usepackage{amssymb}
\usepackage{amstext}
\usepackage{amsthm}
\usepackage{hyperref}
\usepackage{enumerate}
\usepackage[new]{old-arrows}
\usepackage{accents}

\hypersetup{
    colorlinks=true,
    linkcolor=blue,
    citecolor=blue,
    urlcolor=teal
}

\usepackage{colonequals}
\usepackage{enumitem}
\usepackage[all,cmtip]{xy}

\usepackage{tikz} 
\usepackage{tikz-cd}

\newcommand{\Alb}{\textnormal{Alb}}

\newcommand{\Jac}{\text{Jac}}

\newcommand{\et}{\textnormal{\'et}}

\newcommand{\ZZ}{\mathbb{Z}}

\newcommand{\CH}{\textnormal{CH}}

\newcommand{\Pic}{\textnormal{Pic}}

\newcommand{\NS}{\textnormal{NS}}

\newcommand{\pr}{\textnormal{pr}}
\newcommand{\ca}[1]{{\mathcal{#1}}}

\DeclareSymbolFont{bbm}{U}{bbm}{m}{n}
\DeclareSymbolFontAlphabet{\mathbbm}{bbm}
\numberwithin{equation}{section}

\newtheorem{theorem}{Theorem}[section]
\newtheorem{lemma}[theorem]{Lemma}

\newtheorem{corollary}[theorem]{Corollary}

\newtheorem{ques}[theorem]{Question}

\theoremstyle{definition}
\newtheorem{definition}[theorem]{Definition}

\newtheorem{remark}[theorem]{Remark}

\theoremstyle{remark}

\newtheorem{example}[theorem]{Example}
\newtheorem{step}{Step}

\newcommand{\Q}{\mathbb{Q}}

\newcommand{\Z}{\mathbb{Z}}

\newcommand{\bbC}{\mathbf{C}}

\newcommand{\bbF}{\mathbf{F}}

\newcommand{\bbP}{\mathbf{P}}
\newcommand{\bbQ}{\mathbf{Q}}

\newcommand{\tf}{\text{tf}}

\newcommand{\cl}{\text{cl}}

\DeclareMathOperator{\al}{alb}

\newcommand{\Ch}{\operatorname{CH_{0}}}

\DeclareMathOperator{\tor}{tor}

\DeclareMathOperator{\Gal}{Gal}
\DeclareMathOperator{\End}{End}
\DeclareMathOperator{\Hom}{Hom}

\DeclareMathOperator{\Spec}{Spec}

\DeclareMathOperator{\Id}{Id}

\DeclareMathOperator{\coker}{coker}

\DeclareMathOperator{\jac}{Jac}
\DeclareMathOperator{\Gi}{G}

\usepackage[backend=bibtex,style=alphabetic]{biblatex}

\begin{document}
\title[A surface with representable $\Ch$ but no universal zero-cycle]{A surface with representable $\Ch$-group but no universal zero-cycle}

\author{Theodosis Alexandrou} 
\address{Institut für Mathematik, Humboldt-Universität zu Berlin, Rudower Chaussee 25, 10099 Berlin, Deutschland.}
\email{theodosis.alexandrou@hu-berlin.de}

\date{\today} 

\subjclass[2020]{primary 14C25, 14C30, secondary 14J27}

\keywords{algebraic cycles, degenerations of algebraic surfaces, integral Hodge conjecture} 

\begin{abstract}
We introduce a new obstruction to the existence of a universal $0$-cycle on a smooth projective complex variety. Using the geometry of bielliptic surfaces of type $2$, we construct a smooth projective complex surface with representable Chow group of $0$-cycles that admits no universal $0$-cycle. This gives a two-dimensional analogue of Voisin's recent threefold counterexample to a question of Colliot-Thélène. As a further application, we construct a smooth projective complex threefold with torsion canonical bundle for which the integral Hodge conjecture in degree $4$ fails even modulo torsion.
\end{abstract}

\maketitle

\section{Introduction}
Let $X$ be a smooth projective complex variety of dimension $d=\dim X$. Given a base-point $x_{0}\in X(\bbC),$ the universal morphism \begin{equation*}
    \al_{X}\colon X\longrightarrow\Alb(X)
\end{equation*}
to the Albanese variety of $X$ (see \cite[Theorem 1.4.4]{Murre}) induces the so-called Abel--Jacobi map for $0$-cycles on $X$ \begin{equation}
    \label{eq:abel-jacobi-map} \alpha_{X}\colon\Ch(X)_{\hom}\longrightarrow\Alb(X).
\end{equation}
The group homomorphism $\alpha_{X}$ no longer depends on the choice of the base-point and it is known to be surjective and regular. The latter asserts that, for any smooth projective complex variety $Z$ equipped with a base-point $z_{0}\in Z(\bbC)$, and for any codimension-$d$ cycle $[\Gamma]\in\CH^{d}(Z\times X)$, the assignment
\begin{equation}\label{eq:psi}
    \psi_{(Z,[\Gamma])}\colon Z\longrightarrow\Alb(X),\quad z\longmapsto \alpha_{X}\circ[\Gamma]_{\ast}(z-z_{0})
\end{equation}
defines a morphism of algebraic varieties; see \cite[Definition 1.6.1]{Murre}. In fact, the Abel--Jacobi map $\alpha_{X}$ also enjoys the universal property of being initial among all regular homomorphisms from $\Ch(X)_{\hom}$ to abelian varieties; see \cite[Example 1.8(b)]{Murre}.
\par Since the homomorphism $\alpha_{X}$ in \eqref{eq:abel-jacobi-map} is regular and surjective, there exists a codimension-$d$ cycle $[\Gamma]\in\CH^{d}(\Alb(X)\times X)$ and a positive integer $n$ such that the morphism $\psi_{(\Alb(X),[\Gamma])}$ in \eqref{eq:psi}, with respect to the chosen base-point $0_{\Alb(X)}\in\Alb(X)(\bbC)$, is equal to $n\cdot\Id_{\Alb(X)}$; see \cite[Corollary 1.6.3]{Murre}.\par

The present work is concerned with a strengthening of this statement, in the sense of the following property of the variety $X$, first introduced and studied by Voisin; see \cite{Voisin2024CycleCO,Voisin2025}.

\begin{definition}[\cite{voisin_a}, Definition 1.1]\label{def:universal_0-cycle} We say that $X$ admits a \emph{universal $0$-cycle} if there exists a codimension-$d$ cycle $[\Gamma]\in\CH^{d}(\Alb(X)\times X)$ such that the morphism $\psi_{(\Alb(X),[\Gamma])}$ in \eqref{eq:psi}, with respect to the base-point $0_{\Alb(X)}\in\Alb(X)(\bbC)$, is the identity map $\Id_{\Alb(X)}$ on the Albanese variety of $X$. 
\end{definition}

If $X$ is a curve, then the Poincaré divisor on $\Jac(X)\times X$ induces a universal $0$-cycle on $X$. In contrast, Voisin showed in \cite{Voisin2025} that for each $d\geq2$, there exists a smooth projective complex $d$-fold $X$ admitting no universal $0$-cycle; see \cite[Corollary  0.14]{Voisin2025}.\par
In \cite{colliotthélène2025notessurlapplicationdalbanese}, Colliot--Th\'{e}l\`ene considered whether the existence of a universal $0$-cycle persists under additional geometric assumptions on a smooth projective complex variety $X$; see \cite[Question 3.6]{colliotthélène2025notessurlapplicationdalbanese}. In particular, he asked whether this property holds for varieties with representable $\Ch$-group, that is, for which the homomorphism $\alpha_{X}$ in \eqref{eq:abel-jacobi-map} is an isomorphism; see \S \ref{subsec:var_repr_Ch}. Building on earlier ideas of Benoist and Ottem \cite{benoist-ottem}, Voisin subsequently constructed in \cite{voisin_a} a smooth projective complex threefold $X$ with representable $\Ch$-group but admitting no universal $0$-cycle.\par

This leads to the following question, raised by Totaro after Voisin’s talk at the workshop \emph{Hodge Theory and Algebraic Cycles} (Clay Mathematics Institute, September 2025).

\begin{ques}\label{ques:2-dim_example}
Does there exist a smooth projective complex surface $S$ whose Chow group of $0$-cycles $\CH_0(S)$ is representable, but which admits no universal $0$-cycle?
\end{ques}

The main result of the paper answers Question \ref{ques:2-dim_example} in the affirmative.

\begin{theorem}\label{thm:main_result}
Let $E$ be an elliptic curve over $\bbC$. Suppose that either
\begin{enumerate}[label=\textup{(\roman*)}]
    \item\label{it:trivial_end_C} $\End(E)\cong \Z$, or
    \item\label{it:CM_E_C} $E$ has complex multiplication by the maximal order in one of the Heegner fields $\Q(\sqrt{-c})$, where $c\in\{1,2,3,11,19,43,67,163\}$.
\end{enumerate}
Then there exists a smooth projective surface $S$ over $\bbC$ with $\Alb(S)\cong E$ such that $\Ch(S)$ is representable and $S$ does not admit a universal $0$-cycle.
\end{theorem}

Recall that ruled surfaces have representable Chow groups of $0$-cycles and do admit a universal $0$-cycle; this follows directly from the corresponding result for curves.\par 
More generally, let $X$ be a non-uniruled smooth projective surface with geometric genus $p_{g}(X)=0$ and irregularity $q(X)\neq0$. By work of Beauville, the structure of such surfaces is well understood; see \cite[Chapter VI]{beauville_surfaces}. Specifically, $X$ is birational to a quotient \begin{equation*}\label{eq:structure_X}\bigl(C\times D\bigr)/\Gi,\end{equation*}
where $C$ and $D$ are smooth projective curves and $\Gi$ is a finite group acting faithfully on both curves, without fixed points on $C \times D$. Moreover, at least one of the curves is elliptic. Up to exchanging $C$ and $D$, one has $$C/\Gi\cong E,\qquad D/\Gi\cong\bbP^{1}_{\bbC},$$ where $E$ is an elliptic curve. In this situation, the induced fibration \begin{equation}\label{eq:fibration_X}\phi_{X}\colon X\longrightarrow E\end{equation} coincides with the Albanese morphism of $X$.\par
It was shown in \cite{voisin_a} that if the index of the fibration $\phi_X$ in \eqref{eq:fibration_X}, defined as the greatest common divisor of the degrees $\deg(C/E)$, where $C \subset X$ ranges over irreducible curves dominating $E$, is equal to one, then $X$ admits a universal $0$-cycle. In particular, this condition is satisfied when the group $\Gi$ is cyclic; see \cite[Proposition 4.4]{voisin_a}.\par 
Our example in Theorem \ref{thm:main_result} arises from a very general bielliptic surface $S$ of type 2 (see \cite{bielliptic, bielliptic_suwa}) whose Albanese variety $\Alb(S)$ is isomorphic to the prescribed elliptic curve $E$. In this situation, the Albanese fibration $\phi_{S}\colon S\to \Alb(S)\cong E$ has index 2; see Table \ref{tab:bielliptic_coh}. As observed in \cite[$\S 2.2$]{voisin_a}, this condition alone does not suffice to imply that $S$ admits no universal $0$-cycle; see also Example \ref{ex:bielliptic_uni_0_cycle}. The proof of Theorem \ref{thm:main_result} therefore rests on the following geometric result.

\begin{theorem}\label{thm:intro_degeneration_S} Let $k$ be an algebraically closed field of characteristic $\neq2,$ and let $E$ be an elliptic curve over $k.$ Suppose that either \begin{enumerate}[label=\textup{(\roman*)}]\item\label{it:trivial_end} $\End(E)\cong\Z,$ or \item\label{it:CM_E} $E$ has complex multiplication by the maximal order in one of the Heegner fields $\Q(\sqrt{-c}),$ where $c\in\{1,2,3,11,19,43,67,163\}.$\end{enumerate} Set $\Delta:=\Spec k[[t]],$ and fix an algebraic closure $K$ of its function field. Then there exists a regular, flat, projective scheme $\mathcal{S}\to\Delta,$ together with a morphism $$\phi_{\mathcal{S}}\colon\mathcal{S}\longrightarrow E\times\Delta$$ such that the following properties hold:
\begin{enumerate}
    \item\label{it:generic_fibre} The geometric generic fibre $S_{K}$ is a bielliptic surface of type $2$, and the induced morphism $\phi_{S_{K}}\colon S_{K}\to E\times_{k} K$ coincides with the Albanese fibration of $S_{K}$.
    \item\label{it:special_fibre} The special fibre $S_{0}=\sum_{i=1}^{n}R_{i}$ is a reduced simple normal crossings divisor on $\mathcal{S}$ whose dual graph is a chain; that is, each intersection $R_{i}\cap R_{i+1}$ is a smooth irreducible curve, and $R_{i}\cap R_{j}=\varnothing$ whenever $j\notin\{i-1,i,i+1\}$.
     \item\label{it:albanese_S_0} For each $i,$ let $\al_{i}:\Alb(R_{i})\to E$ denote the homomorphism of Albanese varieties induced by the restriction of $\phi_{\mathcal{S}}$ to $R_{i}$. Then the sum morphism $$\sum_{i=1}^{n}\al_{i}\colon\bigoplus_{i=1}^{n}\Alb(R_{i})\longrightarrow E$$ does not admit a section. 
\end{enumerate}  
\end{theorem}
Either condition \ref{it:trivial_end} or \ref{it:CM_E} above implies that the sum morphism in \eqref{it:albanese_S_0} admits no section. As detailed in $\S$\ref{sec:obstruction}, this constitutes the key obstruction to the existence of a universal $0$-cycle on the geometric generic fibre of $\mathcal{S}\to\Delta$; see Theorem \ref{thm:obstruction} and Corollary \ref{cor:obstruction}.\par
We conclude the introduction with an application of Theorem \ref{thm:main_result} to the integral Hodge conjecture.
\begin{corollary}\label{cor:integral_Hodge_S} Let $E$ be an elliptic curve over $\bbC$ such that either \ref{it:trivial_end_C} or \ref{it:CM_E_C} holds for $\End(E)$. Then there exists a bielliptic surface $S$ of type 2 whose Albanese variety $\Alb(S)$ is isomorphic to $E,$ and for which the integral Hodge conjecture for $1$-cycles on $E\times S$ fails modulo torsion. More precisely, there exists an integral Hodge class $$\delta_{S}\in H^{4}(E\times S,\Z)/\tor$$ that is not algebraic.
\end{corollary}
The integral Hodge conjecture for $1$-cycles is known for several important classes of smooth projective complex threefolds. These include all threefolds of Kodaira dimension $-\infty$ and all threefolds $X$ of Kodaira dimension zero with $h^0(X,K_{X})>0$; see \cite{Voisin2006,Totaro2021,Grabowski2004}. For threefolds with non-trivial torsion canonical bundle, however, the conjecture can fail, as shown by Benoist and Ottem \cite{benoist-ottem}; see also \cite{ColliotThelene2019}.\par

Corollary \ref{cor:integral_Hodge_S} gives further counterexamples with torsion canonical bundle, in which the failure persists modulo torsion. This distinguishes our examples from those of Benoist and Ottem, whose obstruction is a non-algebraic torsion class in degree $4$ cohomology. Our proof also uses a different degeneration strategy: we keep the elliptic curve $E$ fixed and degenerate the surface $S$, whereas Benoist and Ottem degenerate the elliptic curve. This approach is motivated by recent work of the author \cite{Alexandrou,alexandrou2025torsionhigherchowcycles} and by work of Schreieder \cite{Schreieder2025}.\par
Failure modulo torsion is also known for other threefolds of Kodaira dimension zero. In particular, Diaz \cite{DiazHumberto2020Otuc} established such a failure for very general Kummer threefolds. His proof uses unramified cohomology, whereas our construction gives an explicit non-algebraic integral Hodge class on a product of an elliptic curve and a bielliptic surface.
\par
The paper is organized as follows. In $\S \ref{sec:preliminaries}$ we review the necessary background material. Section \ref{sec:obstruction} introduces an obstruction to the existence of a universal $0$-cycle on a smooth projective variety $X$ with representable Chow group of $0$-cycles. In $\S \ref{sec:degen_bielliptic}$ we prove Theorem \ref{thm:intro_degeneration_S}. Finally, $\S \ref{sec:final}$ is devoted to the proofs of our main results, including Theorem \ref{thm:main_result} and Corollary \ref{cor:integral_Hodge_S}.

\section{Preliminaries}\label{sec:preliminaries}

\subsection{Notation}\label{subsec:notation} Let $n$ be a positive integer and $A$ an abelian group. We denote by
$$
A[n] := \{a \in A \mid na = 0\}
$$
the $n$-torsion subgroup of $A$. By a slight abuse of notation, for a homomorphism of abelian groups
$\varphi \colon A \to B$ we write
$$
B/A := \coker(\varphi).
$$
For convenience, we also set
$$
A/n := A \otimes_{\mathbb{Z}} \mathbb{Z}/n.
$$

Let $k$ be a field. An \emph{algebraic $k$-scheme} is a separated scheme of finite type over $k$, and a \emph{variety} over $k$ is an integral algebraic $k$-scheme. If $k \subset K$ is a field extension and $X$ is a $k$-scheme, we denote by
$$
X_K := X \times_{\Spec k} \Spec K
$$
the base change of $X$ to $K$.\par
Let $R$ be a discrete valuation ring with algebraically closed residue field $k$, and fix an algebraic
closure $K$ of its fraction field. For a flat $R$-scheme
$\mathcal{X} \to \Spec R,$ we denote by
$$
X_0 := \mathcal{X} \times_R k
\quad\text{and}\quad
X_K := \mathcal{X} \times_R K
$$
the special fibre and the geometric generic fibre, respectively. We say that
$\mathcal{X} \to \Spec R$ is \emph{strictly semistable} if $\mathcal{X}$ is regular, $X_K$ is smooth,
and $X_0$ is a reduced simple normal crossings divisor on $\mathcal{X}$; that is, the irreducible
components of $X_0$ are smooth, and the scheme-theoretic intersection of any $r$ distinct components
of $X_0$ is either empty or smooth and equidimensional of codimension $r$ in $\mathcal{X}$; see \cite[Definition 1.1]{Hartl2001SemistabilityAB}.\par
The \emph{Chow group} of codimension-$i$ cycles on a variety $X/k$ modulo rational equivalence is denoted by $\CH^{i}(X)$ (or $\CH_{\dim X-i}(X)$). The \emph{degree map} for $X/k$ proper is the homomorphism $$\deg\colon \Ch(X)\longrightarrow\Z,\quad \sum_{j=1}^{n}m_{j}\,[x_{j}]\longmapsto \sum_{j=1}^{n}m_{j}\,[k(x_{j}):k],$$ where $k(x_j)$ denotes the residue field of the closed point $x_j\in X$. The subgroup $\Ch(X)_{\hom}$ of nullhomologous $0$-cycles is defined as the kernel of the degree map.

\subsection{Cohomology}\label{subsec:cohomology} Let $X$ be a scheme, let $\ell$ be a prime invertible on $X,$ and fix a positive integer $\nu$. Denote by $\mu_{\ell^{\nu}}$ the \'etale sheaf of $\ell^\nu$-th roots of unity on $X_{\et}.$ For $n\in\Z,$ define the Tate twists
$$\mu^{\otimes n}_{\ell^{\nu}}:=\begin{cases} \mu_{\ell^{\nu}}\otimes\cdots\otimes\mu_{\ell^{\nu}}\quad (n\text{ times}), & n\ge1,\\
\Z/\ell^{\nu}, & n=0,\\
\Hom\bigl(\mu^{\otimes -n}_{\ell^{\nu}}, \Z/\ell^{\nu}\bigr), & n<0.
\end{cases}$$
For a closed immersion $Z\hookrightarrow X,$ the \'etale cohomology with supports in $Z$ is defined by 
\begin{equation}\label{eq:coh_support}
    H^{i}_{Z}\bigl(X,\Z/\ell^{\nu}(n)\bigr):= H^{i}_{Z}\bigl(X_{\et},\mu^{\otimes n}_{\ell^{\nu}}\bigr).
\end{equation}
When $Z=X,$ we simply write 
\begin{equation}\label{eq:etale_cohomology}
    H^{i}_{\et}\bigl(X,\Z/\ell^{\nu}(n)\bigr):=H^{i}_{X}\bigl(X,\Z/\ell^{\nu}(n)\bigr).\end{equation} 

If $X$ is an algebraic scheme over an algebraically closed field $k$, we suppress the Tate twist $(n)$ from the notation. \par

The twisted cohomology theory \eqref{eq:coh_support} (for $X$ an algebraic $k$-scheme) forms a Poincar\'e duality theory with supports in the sense of Bloch--Ogus; see \cite[Definition (1.3), Example (2.1)]{Bloch-Ogus1974}. In particular, let $X$ and $Y$ be smooth varieties over a field $k,$ with $X$ proper of dimension $d=\dim X.$ Then any algebraic cycle $$[\Gamma]\in \CH^{c}(X\times Y)$$ induces a correspondence action 
 \begin{equation}\label{eq:action-correspondences}
     [\Gamma]_{\ast}\colon H^{i}_{\et}\bigl(X,\Z/\ell^{\nu}(n)\bigr)\longrightarrow H^{i+2c-2d}_{\et}\bigl(Y,\Z/\ell^{\nu}(n+c-d)\bigr)
 \end{equation}
 given explicitly by $$ \gamma\longmapsto q_{\ast}\bigl(p^{\ast}\gamma\ \cup\ \cl_{X\times Y}^c([\Gamma])\bigr),$$
where $$\cl^{c}_{X\times Y}\colon \CH^{c}(X\times Y)\longrightarrow H^{2c}_{\et}\bigl(X\times Y,\Z/\ell^{\nu}(c)\bigr)$$ is the cycle class map, and $$p\colon X\times_k Y\longrightarrow X,\qquad q\colon X\times_k Y\longrightarrow Y$$ are the natural projection morphisms.\par
 For a closed immersion $Z\hookrightarrow X,$ we define the $\ell$-adic \'etale cohomology with supports by
\begin{equation}\label{eq:coh_support_Z_l}
    H^{i}_{Z}\bigl(X,\Z_\ell(n)\bigr):= \varprojlim_{\nu}\ H^{i}_{Z}\bigl(X,\Z/\ell^{\nu}(n)\bigr).
\end{equation}
When $Z=X,$ we write
\begin{equation}\label{eq:etale_cohomology_Z_l}
     H^{i}_{\et}\bigl(X,\Z_\ell(n)\bigr):= H^{i}_{X}\bigl(X,\Z_\ell(n)\bigr).
\end{equation}
\par The $\ell$-adic \'etale cohomology groups defined in \eqref{eq:etale_cohomology_Z_l} have good formal properties only when the groups $H_{\et}^i(X,\Z/\ell^\nu(n))$ are finite. In this case, they agree with Jannsen's continuous $\ell$-adic \'etale cohomology \cite{Jannsen1988}. This finiteness condition holds, for example, when $X$ is a variety over an algebraically closed field $k$. By the proper base change theorem, it also holds when $X$ is proper over a strictly Henselian local ring $R$ (e.g.\ $R=k[[t]]$); see \cite[Corollary VI.2.7]{milne_etale}.

\subsection{Varieties with representable Chow group of 0-cycles}\label{subsec:var_repr_Ch}

There are several equivalent characterizations of varieties with representable Chow group of $0$-cycles; we refer to \cite[\S 10.1]{Voisin_2003} for a detailed account.

Let $X$ be a smooth projective variety over $\bbC$. For a positive integer $n$, denote by $X^{(n)}$ the $n$-fold symmetric product of $X$. Its closed points correspond to effective $0$-cycles of degree $n$ on $X$. There is a natural map
\begin{equation*}
    \rho_n \colon X^{(n)}(\bbC) \times X^{(n)}(\bbC) \longrightarrow \Ch(X)_{\hom},
    \quad ([Z_1],[Z_2]) \longmapsto [Z_1 - Z_2].
\end{equation*}
We say that $\Ch(X)$ is \emph{representable} if $\rho_n$ is surjective for all sufficiently large $n$. This condition is equivalent to the finite-dimensionality of $\Ch(X)$; see \cite[Proposition 10.10]{Voisin_2003}.\par

By a theorem of Roitman \cite[Theorem 4]{Roitman_1972}, $\Ch(X)$ is representable if and only if the Abel--Jacobi map for $0$-cycles
$$
\alpha_X \colon \Ch(X)_{\hom} \longrightarrow \Alb(X)
$$
is an isomorphism; see also \cite[Theorem 10.11]{Voisin_2003}. This is further equivalent to the condition that $\Ch(X)$ be supported on any smooth complete intersection curve $C \subset X$ cut out by ample hypersurfaces; see \cite[Proposition 10.12]{Voisin_2003}. In particular, the Bloch--Srinivas decomposition of the diagonal argument \cite{bloch-srinivas} implies that
$$
H^{0}(X,\Omega_{X}^{i}) = 0 \quad \text{for all } i \geq 2.
$$

We record the following structural consequence.

\begin{lemma}\label{lem:structure_varieties_rep_CH_0}
Let $X$ be a smooth projective variety over $\bbC$ satisfying
$$
H^{0}(X,\Omega_{X}^{i}) = 0 \ \text{for all } i \geq 2,
\qquad
H^{0}(X,\Omega_{X}^{1}) \neq 0.
$$
Then the Albanese morphism $\al_X \colon X \to \Alb(X)$ factors as
\begin{equation}\label{eq:factorization}
X \xrightarrow{\ \phi_X\ } C \xrightarrow{\ \psi\ } \Alb(X),
\end{equation}
where $\phi_X$ is a surjective morphism with connected fibres onto a smooth projective curve $C$, and $\psi$ is finite. Moreover, $\psi$ induces an isomorphism
$$
\jac(C) \cong \Alb(X).
$$
\end{lemma}

\begin{proof}
See \cite[Lemma 1.5 and Remark 1.6]{voisin_a}.
\end{proof}

\subsection{Bielliptic surfaces}\label{subsec:bielliptic} Throughout this subsection, the base field $k$ is assumed to be algebraically closed.\par 
Bielliptic surfaces constitute one of the classes of smooth projective surfaces of Kodaira dimension zero in the Enriques--Kodaira classification; see \cite{beauville_surfaces,BombieriMumford1977}. Every bielliptic surface $S$ admits a presentation \begin{equation*}\label{eq:presentation_S}
    S\cong\bigl(E\times_k F\bigr)/\Gi,\end{equation*}
where $E$ and $F$ are elliptic curves, and $\Gi$ is a finite group acting freely on $E\times_k F.$ The action is by translations on $E$ and by automorphisms on $F$. Moreover, one has an isomorphism $$F/\Gi\cong\bbP^{1}_k.$$\par
Bielliptic surfaces were classified by Bagnera--De Franchis into seven types; see \cite{bielliptic,bielliptic_suwa}. The classification is summarized in Table \ref{tab:bielliptic}.
\renewcommand{\arraystretch}{1.5} % Increase row height
\begin{table}[h]
\centering
\caption{Classification of Bielliptic Surfaces (Bagnera--de Franchis)}
\label{tab:bielliptic}
\begin{tabular}{c c l c}
\hline
\textbf{Type} & \textbf{Group $\Gi$} & \textbf{$\NS(S)_{\tor}$} & \textbf{Order of $K_S$ in $\Pic(S)$} \\
\hline
1 & $\mathbb{Z}/2$ & $\Z/2\times\Z/2$ & 2 \\
2 & $\mathbb{Z}/2 \times \mathbb{Z}/2$ & $\Z/2$ & 2 \\
3 & $\mathbb{Z}/4$ & $\Z/2$ & 4 \\
4 & $\mathbb{Z}/4\times \mathbb{Z}/2$ & $0$ & 4 \\
5 & $\mathbb{Z}/3$ & $\Z/3$ & 3 \\
6 & $\mathbb{Z}/3 \times \mathbb{Z}/3$ & $0$ & 3 \\
7 & $\mathbb{Z}/6$ & $0$ & 6 \\
\hline
\end{tabular}
\end{table}
\par The Albanese morphism of $S$ coincides with the projection $$\phi_{S}\colon S\longrightarrow E/\Gi,$$ whose fibres are elliptic curves isomorphic to $F$. The induced Abel--Jacobi map on $0$-cycles,
$$(\phi_{S})_{\ast}\colon\Ch(S)_{\hom}\longrightarrow E/\Gi$$
is an isomorphism; see \cite[Example (2)]{bloch-srinivas}. Equivalently, the Chow group of $0$-cycles on a bielliptic surface is representable.\par
There is also a natural elliptic fibration $$\psi_{S}\colon S\longrightarrow (F/\Gi)\cong\bbP^{1}_{k}.$$ 
This fibration admits multiple fibres corresponding to the branch points of the quotient morphism $F \to F/\Gi$, with multiplicities equal to the corresponding ramification indices.\par
Let $f$ and $e$ denote the classes in $\NS(S)$ of the smooth fibres of $\phi_S$ and $\psi_S$, respectively. A direct computation shows \begin{equation}\label{eq:intersection_e_f}e^2=0,\quad,f^2=0,\quad e\cdot f=|\Gi|.\end{equation}

The structure of the Néron--Severi group $\NS(S)$ is given in Table \ref{tab:bielliptic_coh}; see \cite[Table 2]{Serrano1990}.

\renewcommand{\arraystretch}{1.5}
\begin{table}[h]
\centering
\caption{Structure of $\NS(S)$}
\label{tab:bielliptic_coh}
\begin{tabular}{c c c}
\hline
\textbf{Type} & \textbf{$\NS(S)$} & \textbf{Index of $\phi_S$} \\
\hline
1 & $\Z[\tfrac{1}{2}e] \oplus \Z[f] \oplus \Z/2 \oplus \Z/2$ & $1$ \\
2 & $\Z[\tfrac{1}{2}e] \oplus \Z[\tfrac{1}{2}f] \oplus \Z/2$ & $2$ \\
3 & $\Z[\tfrac{1}{4}e] \oplus \Z[f] \oplus \Z/2$ & $1$ \\
4 & $\Z[\tfrac{1}{4}e] \oplus \Z[\tfrac{1}{2}f]$ & $2$ \\
5 & $\Z[\tfrac{1}{3}e] \oplus \Z[f] \oplus \Z/3$ & $1$ \\
6 & $\Z[\tfrac{1}{3}e] \oplus \Z[\tfrac{1}{3}f]$ & $3$ \\
7 & $\Z[\tfrac{1}{6}e] \oplus \Z[f]$ & $1$ \\
\hline
\end{tabular}
\end{table}

\par The last column records the index of $\phi_S$, defined as the greatest common divisor of the degrees
$$
\deg\bigl(C \longrightarrow E/\Gi\bigr),
$$
where $C \subset S$ ranges over irreducible curves dominating $E/\Gi$. This index is computed directly from \eqref{eq:intersection_e_f} together with the structure of $\NS(S)$.\par

By \cite[Proposition 2.7]{voisin_a}, if the index of $\phi_S$ equals $1$, then $S$ admits a universal $0$-cycle. When the index is greater than $1$, the existence of a universal $0$-cycle remains unclear; see \cite[\S 2.2]{voisin_a} and Example \ref{ex:bielliptic_uni_0_cycle}. From this perspective, the potentially interesting cases are types $2$, $4$, and $6$.

\section{An obstruction to the existence of a universal 0-cycle}\label{sec:obstruction}
The following result is the key input for the proof of Theorem \ref{thm:main_result}.

\begin{theorem}
\label{thm:obstruction}
Let $k$ be an algebraically closed field and let $C$ be a smooth projective curve over $k$. Set $\Delta := \Spec k[[t]]$, and fix an algebraic closure $K$ of its function field. 

Let $p \colon \mathcal{X} \to \Delta$ be a flat, projective morphism of relative dimension $d$ with regular total space, and let $\phi_{\mathcal{X}} \colon \mathcal{X} \to C \times_k \Delta$ be a surjective morphism.
Denote by $X_0$ the special fibre of $p$, by
$\phi_{X_0} \colon X_0 \to C$ the induced morphism, and by $X_{K
}$ the geometric generic fibre of $p$. Write $\phi_{X_K} \colon X_K \to C_K$
for the base change of $\phi_{\mathcal{X}}$ to $K$.\par
We further assume the following:
\begin{enumerate}
    \item \label{it:X_0} $X_{0}=\sum_{i=1}^{n}X_{0i}$ is a reduced simple normal crossings divisor on $\mathcal{X}$ whose dual graph is a chain. Precisely, each intersection $X_{0i}\cap X_{0i+1}$ is smooth and irreducible of dimension $d-1,$ and $X_{0i}\cap X_{0j}=\varnothing$ whenever $j\notin\{i-1,i,i+1\}$.
    \item \label{it:correspondence} 
    There exists a correspondence
$$
[\Gamma] \in \CH^{d}(C \times_k X_{K})
$$
inducing a splitting of the pushforward map
$$
(\phi_{X_K})_\ast\colon \Ch(X_{K})_{\hom}\longrightarrow\jac(C_{K}).
$$
    
\end{enumerate}
\par 
For each $i,$ let $$\al_{i}\colon\Alb(X_{0i})\longrightarrow \jac(C)$$ be the homomorphism of Albanese varieties induced by the restriction of $\phi_{X_0}$ to $X_{0i}$. Then the induced sum morphism $$\sum_{i=1}^{n}\al_{i}\colon\bigoplus_{i=1}^{n}\Alb(X_{0i})\longrightarrow\jac(C)$$
admits a section. 
\end{theorem}
\begin{proof}
By assumption \eqref{it:correspondence}, there exists a correspondence $$[\Gamma]\in\CH^{d}(C\times_{k}X_{K})$$
such that the induced endomorphism $$(\phi_{X_{K}})_{\ast}\circ [\Gamma]_{\ast}\in\End\bigl(\jac(C_{K})\bigr)$$
is the identity.\par
We begin with the following reduction step.
\begin{step}\label{step:reduction} 
After performing a finite ramified base change $\Delta' \to \Delta$ and replacing $\ca X$ by a suitable modification of the fiber product $\ca X \times_{\Delta} \Delta',$
we may assume that the correspondence $[\Gamma]$ lifts to a class in
$$
\CH^{d}(C\times_{k}\mathcal{X}).
$$
\begin{proof}[Proof of Step \ref{step:reduction}] Let $\mathcal{K}$ be the directed system of intermediate fields $$k((t))\subset E\subset K $$ consisting of finite field extensions of $k((t)),$ ordered by inclusion. Since Chow groups commute with filtered colimits, we have $$\CH^{d}(C\times_{k}X_{K})=\varinjlim_{E\in\mathcal{K}}\CH^{d}(C\times_{k}X_{E}).$$ Hence there exists a finite field extension $L/k((t))$ such that $[\Gamma]$ is the image of a class $$[\Gamma_{L}]\in\CH^{d}(C\times_{k}X_{L}).$$ \par

Let $R\subset L$ be the integral closure of $k[[t]]$ in $L$. Then $R$ is a complete discrete valuation ring with residue field $k$ and fraction field $L.$ In particular, there exists an isomorphism $R\cong k[[\mu]]$ under which the inclusion $k[[t]] \hookrightarrow R$ is given by $$t \longmapsto \mu^m,$$ for some $m\geq1.$\par
Set $\Delta':=\Spec k[[\mu]],$ and let $p'\colon\mathcal{X}'\to\Delta'$ be the base change of $p\colon\mathcal{X}\to\Delta$ along the induced morphism $\Delta'\to\Delta$.\par
In general, $\mathcal{X}'$ need not be regular. As in Hartl's construction (cf. the proof of \cite[Proposition 2.2]{Hartl2001SemistabilityAB}), by successively blowing up non-Cartier divisors in the special fibre one obtains a strictly semistable model $$\widetilde{p}\colon\widetilde{\mathcal{X}}\longrightarrow\Delta'.$$ Since $$X_{0}=\bigcup_{i=1}^{n}X_{0i}$$ is a chain, the special fibre of $\widetilde{p},$ \begin{equation*}\widetilde{X}_{0}=X_{01}\cup \bigl(\bigcup_{j=1}^{m-1}R_{1,j}\bigr)\cup X_{02}\cup \bigl(\bigcup_{j=1}^{m-1}R_{2,j}\bigr)\cup \dots\cup X_{0n}\end{equation*} is again a chain. The newly introduced components $R_{i,j}$ are $\bbP^{1}_{k}$-bundles over $$Y_{i,i+1}:=X_{0i}\cap X_{0i+1},$$ with projection morphisms $\pi_{i,j}\colon R_{i,j}\to Y_{i,i+1}$. Moreover, the intersections $$X_{0i}\cap R_{i,1}, \quad R_{i,m-1}\cap X_{0i+1}, \quad R_{i,j}\cap R_{i,j+1}$$ are all canonically isomorphic to $Y_{i,i+1},$ and hence define sections of the corresponding $\bbP^{1}_{k}$-bundles; see also \cite[Remark 5.2]{Alexandrou}.

The class $[\Gamma_{L}]$ is defined on the generic fibre of the family $$C\times_{k}\widetilde{\mathcal{X}}\longrightarrow\Delta'.$$ Taking the closures of the irreducible components of its support yields a class in $$\CH^{d}(C\times_{k}\mathcal{\widetilde{X}})$$ whose restriction to the geometric generic fibre coincides with $[\Gamma]$.

Let $$\phi_{\widetilde{\mathcal{X}}}\colon\widetilde{\mathcal{X}}\longrightarrow C\times_{k}\Delta'$$ be the morphism induced by $\phi_{\mathcal{X}}.$ Write $\Alb(\widetilde{X}_{0})$ (resp.\ $\Alb(X_{0})$) for the direct sum of the Albanese varieties of the irreducible components of $\widetilde{X}_0$ (resp.\ $X_0$). 
\par
It remains to show that if the sum morphism $\Alb(\widetilde X_0) \to \jac(C)$ admits a section, then the same holds for $\Alb(X_0) \to \jac(C)$. It suffices to verify that the natural morphism $$\Alb(\widetilde X_0) \longrightarrow \jac(C)$$ factors through $\Alb(X_0)\to\jac(C)$. To this end, note that each $R_{i,j}$ is a $\bbP^{1}_{k}$-bundle over $Y_{i,i+1}\subset X_{0i}.$ Let $$\al_{i,j}\colon\Alb(R_{i,j})\longrightarrow\jac(C)$$ be the homomorphism of Albanese varieties induced by restricting $\phi_{\widetilde{\mathcal{X}}}$ to $R_{i,j}$. Since $\al_{i,j}$ factors as $$\Alb(R_{i,j}) \longrightarrow \Alb(X_{0i}) \longrightarrow \jac(C)$$ for each $i,j$, it follows that the natural morphism $$\Alb(\widetilde{X}_0) = \Alb(X_0) \oplus \bigoplus_{i,j} \Alb(R_{i,j}) \longrightarrow \jac(C)$$ factors through $\Alb(X_0)\to\jac(C),$ as claimed. This completes the reduction step.
\end{proof}    
\end{step}
By Step \ref{step:reduction}, we may therefore assume that $[\Gamma]$ extends to a class $$[\Gamma_{\mathcal{X}}]\in\CH^{d}(C\times_{k}\mathcal{X}).$$ \par
For $1\leq i\leq n,$ let $$j_{i}\colon X_{0i}\longhookrightarrow\mathcal{X}$$ denote the closed immersion, and define $$[\Gamma_{i}]:=(\Id\times j_{i})^{\ast}[\Gamma_{\mathcal{X}}]\in\CH^{d}(C\times_{k}X_{0i}).$$ Let $$\alpha_{i}\colon\Ch(X_{0i})_{\hom}\longrightarrow\Alb(X_{0i})$$ be the Abel--Jacobi homomorphism on $0$-cycles, and set $$\beta_{i}:=\alpha_{i}\circ[\Gamma_{i}]_{\ast}\in\Hom\bigl(\jac(C),\Alb(X_{0i})\bigr).$$\par
We aim to show that the endomorphism 
\begin{equation}\label{eq:end}
\Phi\colon \jac(C)\overset{(\beta_{i})_{i}}{\longrightarrow}\bigoplus_{i=1}^{n}\Alb(X_{0i})\overset{\sum_{i=1}^{n}\al_{i}}{\longrightarrow} \jac(C)
\end{equation}
is the identity. \par
Fix a prime $\ell$ invertible in $k,$ and consider the composite \begin{equation}\label{eq:Psi_map}\Psi\colon H^{1}_{\et}\bigl(C\times\Delta,\Z_{\ell}\bigr)\overset{[\Gamma_{\mathcal{X}}]_{\ast}}{\longrightarrow}H^{2d-1}_{\et}\bigl(\mathcal{X},\Z_{\ell}\bigr)\overset{(\phi_{\mathcal{X}})_{\ast}}{\longrightarrow}H^{1}_{\et}\bigl(C\times\Delta,\Z_{\ell}\bigr).\end{equation}
A priori, the correspondence action \eqref{eq:action-correspondences} is defined only for algebraic cycles on a product $X\times Y$, where $X$ and $Y$ are smooth varieties and $X$ is proper. However, since the \'etale cohomology groups \eqref{eq:etale_cohomology} commute with filtered inverse limits of schemes with affine transition morphisms (see \cite[p. 88, III.1.16]{milne_etale}), the correspondence action induced by $[\Gamma_{\mathcal{X}}]_{\ast}$ remains well-defined in the present setting. The same argument applies to the pushforward map $(\phi_{\mathcal{X}})_{\ast}$ appearing in \eqref{eq:Psi_map}.
\par
Let $$j_{K}\colon C_{K}\longhookrightarrow C\times\Delta,\qquad j_{k}\colon C \longhookrightarrow C\times\Delta$$ denote the inclusions of the geometric generic fibre and special fibre, respectively.\par
By the smooth and proper base change theorem for $\ell$-adic \'etale cohomology, the pullback maps 
\begin{align}\label{eq:pullback}
\begin{split}
(j_{K})^{\ast}\colon &H^{1}_{\et}\bigl(C\times\Delta,\Z_\ell\bigr)\longrightarrow H^{1}_{\et}\bigl(C_{K},\Z_\ell\bigr),\\
(j_{k})^{\ast}\colon &H^{1}_{\et}\bigl(C\times\Delta,\Z_\ell\bigr)\longrightarrow H^{1}_{\et}\bigl(C,\Z_\ell\bigr),
\end{split}
\end{align} 
are isomorphisms; see \cite[Corollary VI.4.2]{milne_etale}. By assumption \eqref{it:correspondence}, the endomorphism $(\phi_{X_{K}})_{\ast}\circ[\Gamma]_{\ast}\in\End(\jac(C_{K}))$ is the identity.
The diagram \begin{equation*}\begin{tikzcd}
{H^{1}_{\et}\bigl(C\times\Delta,\Z_\ell\bigr)} \arrow[r, "\Psi"] \arrow[d, "(j_{K})^{\ast}"]    & {H^{1}_{\et}\bigl(C\times\Delta,\Z_\ell\bigr)} \arrow[d, "(j_{K})^{\ast}"] \\
{H^{1}_{\et}\bigl(C_{K},\Z_\ell\bigr)} \arrow[r, "{(\phi_{X_{K}})_{\ast}\circ[\Gamma]_{\ast}}"] & {H^{1}_{\et}\bigl(C_{K},\Z_\ell\bigr)} \end{tikzcd}
\end{equation*}
clearly commutes, with the vertical arrows given by the first isomorphism in \eqref{eq:pullback}. It follows that $\Psi$ is the identity map.\par
Next, observe that for any class $\gamma\in H^{2d-1}_{\et}(\mathcal{X},\Z_\ell)$ we have the identities \begin{equation}\label{eq:projection_formula}(j_{k})^{\ast}(\phi_{\mathcal{X}})_{\ast}(\gamma)=(\phi_{\mathcal{X}})_{\ast}(\gamma)\cup[C]=(\phi_{\mathcal{X}})_{\ast}\bigl(\gamma\ \cup\ [X_{0}]\bigr)=\sum_{i=1}^{n}(\phi_{\mathcal{X}})_{\ast}\bigl(\gamma\ \cup\ [X_{0i}]\bigr),\end{equation}
where the second equality follows from the projection formula; cf. \cite[(1.3.3) and Example (2.1)]{Bloch-Ogus1974}.\par
All operations in \eqref{eq:projection_formula} are compatible with $\ell$-adic \'etale cohomology with support \eqref{eq:coh_support_Z_l}. In particular, for each $i$ one has  $$\gamma\ \cup\ [X_{0i}]\in H^{2d+1}_{X_{0i}}\bigl(\mathcal{X},\Z_\ell\bigr),$$ and $(\phi_{\mathcal{X}})_{\ast}$ denotes the pushforward map $$H^{2d+1}_{X_{0i}}\bigl(\mathcal{X},\Z_\ell\bigr)\longrightarrow H^{3}_{C}\bigl(C\times_{k}\Delta,\Z_\ell\bigr)$$ induced by $\phi_{\mathcal{X}}$.\par
Let $$\phi_{0i}\colon X_{0i}\longrightarrow C$$ denote the restriction of $\phi_{\mathcal{X}}$ to the irreducible component $X_{0i}\subset X_{0}$. By Gabber's absolute purity theorem (see, for example, \cite[Theorem 2.3.1]{ColliotSkoro2021}), identity \eqref{eq:projection_formula} can be rewritten as \begin{equation}\label{eq:projection_formula'}(j_{k})^{\ast}(\phi_{\mathcal{X}})_{\ast}(\gamma)=\sum_{i=1}^{n}(\phi_{0i})_{\ast}(\gamma|_{X_{0i}})\in H^{1}_{\et}\bigl(C,\Z_\ell\bigr),\end{equation}
where $(\phi_{0i})_{\ast}\colon H^{2d-1}_{\et}(X_{0i},\Z_\ell)\to H^{1}_{\et}(C,\Z_\ell)$ is the natural pushforward induced by $\phi_{0i},$ and $\gamma|_{X_{0i}}:=(j_{i})^{\ast}(\gamma)\in H^{2d-1}_{\et}(X_{0i},\Z_{\ell})$.\par
We claim that the diagram
\begin{align} \label{diagram:specialization}
\begin{split}
\xymatrixcolsep{1.4pc}
\xymatrix{
H^1_\et\bigl(C \times \Delta, \ZZ_\ell\bigr) \ar[r]^-{[\Gamma_{\mathcal{X}}]_\ast}\ar[d]_-\wr^-{j_k^\ast} & H^{2d-1}_\et\bigl(\ca X, \ZZ_\ell\bigr) \ar[r]^-{(\phi_{\ca X})_{\ast}} \ar[d] & H^1_\et\bigl(C \times \Delta, \ZZ_\ell\bigr) \ar[d]_-\wr^-{j_k^\ast} \\
H^1_\et\bigl(C,\ZZ_\ell\bigr) \ar[r]^-{([\Gamma_{i}]_\ast)_i} & \bigoplus_{i = 1}^n H^{2d-1}_\et\bigl(X_{0i},\ZZ_\ell\bigr) \ar[r]^-{\sum_{i=1}^n (\phi_{0i})_{\ast}} & H^1_\et\bigl(C, \ZZ_\ell\bigr)
}
\end{split}
\end{align}
commutes, where the vertical map in the middle sends $\gamma \in H^{2d-1}_\et(\ca X,\ZZ_\ell)$ to $$(\gamma|_{X_{0i}})_i\in\bigoplus_{i=1}^{n}H^{2d-1}_\et\bigl(X_{0i},\ZZ_\ell\bigr).$$ Indeed, the commutativity of the left square is immediate, and that of the right square follows from identity \eqref{eq:projection_formula'}. Hence the bottom row of \eqref{diagram:specialization} is the identity.\par 
Moreover, this map agrees with the homomorphism $T_{\ell}\Phi$ induced by $\Phi$ (defined in \eqref{eq:end}) on the Tate module $$T_{\ell}\jac(C)\coloneqq \varprojlim_{\nu}\jac(C)[\ell^{\nu}].$$
Consequently, $$T_{\ell}\Phi=\Id\in\End_{\Z_{\ell}}\bigl(T_{\ell}\jac(C)\bigr).$$ Since the endomorphism ring of an abelian variety is torsion-free, it follows from \cite[Theorem 12.5]{milne1986abelian} that the natural homomorphism $$\End\bigl(\jac(C)\bigr) \longrightarrow \End_{\Z_{\ell}}\bigl(T_\ell\jac(C)\bigr),\quad f\longmapsto T_{\ell}f$$ is injective. We therefore conclude that $$\Phi=\Id\in\End\bigl(\jac(C)\bigr).$$ This completes the proof of the theorem.\end{proof}

\begin{corollary}\label{cor:obstruction}
We adopt the notation and assumptions of Theorem \ref{thm:obstruction}, replacing condition \eqref{it:correspondence} with the following:
\begin{enumerate}
    \item[\((2')\)]\label{it:correspondence-prime} There exists an algebraically closed field extension $F/K$ and a correspondence
$$
[\Gamma] \in \CH^{d}(C \times_k X_{F})
$$
inducing a splitting of the pushforward map
$$
(\phi_{X_F})_\ast\colon \Ch(X_{F})_{\hom}\longrightarrow\jac(C_{F}).
$$
\end{enumerate}
Assume moreover that $\End(\jac(C))$ is a principal ideal domain (PID). Then the sum morphism $$\sum_{i=1}^{n}\al_{i}\colon\bigoplus_{i=1}^{n}\Alb(X_{0i})\longrightarrow\jac(C)$$
admits a section. 
\end{corollary}
\begin{proof}
By assumption (\hyperref[it:correspondence-prime]{2$'$}), there exists an algebraically closed field extension $F/K$ and a correspondence $$[\Gamma] \in \CH^{d}(C \times_k X_{F})$$ such that the composite morphism $$(\phi_{X_F})_\ast\circ[\Gamma]_{\ast}\in\End\bigl(\jac(C_{F})\bigr)$$ is the identity. It suffices to show that this implies \eqref{it:correspondence} in Theorem \ref{thm:obstruction}.\par
Suppose, for the sake of contradiction, that $\Id_{\jac(C_{K})}$ is not in the image of the homomorphism 
\begin{equation}\label{eq:hom_corr}\CH^{d}(C \times_k X_{K})\longrightarrow\End\bigl(\jac(C_{K})\bigr),\quad [\widetilde\Gamma]\longmapsto (\phi_{X_{K}})_{\ast}\circ[\widetilde\Gamma]_{\ast}.\end{equation} 
The pushforward map $$(\phi_{X_{K}})_{\ast}\colon \Ch(X_{K})_{\hom}\longrightarrow\jac(C_{K})$$ is regular and surjective. Consequently, by \cite[Corollary 1.6.3]{Murre}, the image of \eqref{eq:hom_corr} is non-zero. Recall that every endomorphism of $\jac(C_K)$ is induced by a correspondence on $C_{K}\times C_{K},$ and that \eqref{eq:hom_corr} is compatible with the action of $\CH^{1}(C_{K}\times C_{K})$. It follows that its image is a non-zero ideal $$\mathfrak{a}\subset\End(\jac(C_K)).$$ Since $\End(\jac(C_K))$ is a PID, we can write $\mathfrak{a}=(f)$ for some isogeny $f\in\End(\jac(C_K))$ with $m:=\deg(f)>1$. Thus, for every $$ [\widetilde\Gamma]\in\CH^{d}(C \times_k X_{K}),$$ the induced endomorphism $(\phi_{X_K})_\ast\circ[\widetilde\Gamma]_{\ast}$ has a non-trivial kernel on the $m$-torsion subgroup $$\jac(C_{K})[m]=\jac(C_{F})[m].$$ By \cite[Théorème 3.11]{lecomte}, the base-change homomorphism $$\CH^{d}(C \times_k X_{K})/m\longrightarrow \CH^{d}(C \times_k X_{F})/m$$ is an isomorphism. Hence there exists a cycle $[\widetilde\Gamma]\in\CH^{d}(C \times_k X_{K})$ whose base change $[\widetilde{\Gamma}_{F}]$ agrees with $[\Gamma]$ modulo $m$. It follows that the endomorphism $(\phi_{X_K})_\ast\circ[\widetilde\Gamma]_{\ast}$ restricts to the identity on $\jac(C_{K})[m]$, contradicting the previous observation.\par
This contradiction completes the proof.
\end{proof}

\begin{remark}\label{rem:base_fibration}
In the setting of Theorem \ref{thm:obstruction} (resp. Corollary \ref{cor:obstruction}), the assumption that $C$ is a curve is inessential. More generally, one may replace $C$ by an arbitrary smooth projective variety $B,$ the Jacobian $\jac(C)$ by the Albanese variety $\Alb(B)$ of $B,$ and $[\Gamma]$ by an algebraic correspondence in $$\CH^{d}(\Alb(B)\times_k X_{K}).$$ \par
We restrict to the case where $C$ is a curve only because this situation naturally arises for smooth projective varieties with representable $\Ch$-group. Indeed, if $X$ is such a variety, then applying Stein factorization to its Albanese morphism $$\al_{X}\colon X\longrightarrow\Alb(X)$$ produces a fibration $\phi_X \colon X \to C$ onto a smooth projective curve $C$ with $\jac(C) \cong \Alb(X)$; see Lemma \ref{lem:structure_varieties_rep_CH_0}.
\end{remark}

The following result is a further generalization of Theorem \ref{thm:obstruction}. We state it without proof, as the argument proceeds along the same lines. In particular, it is useful for constructing examples over $\bar{\bbQ}$.

\begin{theorem}\label{thm:obstruction'}
Let $R$ be a discrete valuation ring with algebraically closed residue field $k$, and set $\Delta := \Spec R$. Fix an algebraic closure $K$ of the fraction field of $R$.

    Let $\pi\colon \ca C\to\Delta$ be a smooth, projective morphism of relative dimension $1$ with geometrically connected fibres. Denote by $C_0$ the special fibre of $\pi$ and by $C_K$ its geometric generic fibre.

Let $p \colon \mathcal{X} \to \Delta$ be a flat, projective morphism of relative dimension $d$ with regular total space, and let $\phi_{\mathcal{X}} \colon \mathcal{X} \to \ca C $ be a surjective morphism.
Denote by $X_0$ the special fibre of $p$ and by
$\phi_{X_0} \colon X_0 \to C_0$ the induced morphism. Let $X_{K
}$ be the geometric generic fibre of $p$, and write $\phi_{X_K} \colon X_K \to C_K$
for the base change of $\phi_{\mathcal{X}}$ to $K$.\par
Assume the following:
\begin{enumerate}
    \item \label{it:X_0'} $X_{0}=\sum_{i=1}^{n}X_{0i}$ is a reduced simple normal crossings divisor on $\mathcal{X}$ whose dual graph is a chain; see \eqref{it:X_0}.
    \item \label{it:correspondence_2} 
    There exists a correspondence
$$
[\Gamma] \in \CH^{d}(C_K \times_K X_{K})
$$
inducing a splitting of the pushforward map
$$
(\phi_{X_K})_\ast\colon \Ch(X_{K})_{\hom}\longrightarrow\jac(C_{K}).
$$
    
\end{enumerate}
\par 
For each $i,$ let $$\al_{i}\colon\Alb(X_{0i})\longrightarrow \jac(C_0)$$ be the homomorphism of Albanese varieties induced by the restriction of $\phi_{X_0}$ to $X_{0i}$. Then the induced sum morphism $$\sum_{i=1}^{n}\al_{i}\colon\bigoplus_{i=1}^{n}\Alb(X_{0i})\longrightarrow\jac(C_0)$$
admits a section. 
\end{theorem}

We also record the corresponding generalization of Corollary \ref{cor:obstruction}. The proof is analogous and is therefore omitted.

\begin{corollary}\label{cor:obstruction'}
We adopt the notation and assumptions of Theorem \ref{thm:obstruction'}, replacing condition \eqref{it:correspondence_2} with the following:
\begin{enumerate}
    \item[\((2')\)]\label{it:correspondence-prime_2} There exists an algebraically closed field extension $F/K$ and a correspondence
$$
[\Gamma] \in \CH^{d}(C_F \times_F X_{F})
$$
inducing a splitting of the pushforward map
$$
(\phi_{X_F})_\ast\colon \Ch(X_{F})_{\hom}\longrightarrow\jac(C_{F}).
$$
\end{enumerate}
Assume moreover that $\End(\jac(C_K))$ is a principal ideal domain (PID). Then the sum morphism $$\sum_{i=1}^{n}\al_{i}\colon\bigoplus_{i=1}^{n}\Alb(X_{0i})\longrightarrow\jac(C_0)$$
admits a section. 
\end{corollary}

\section{A degeneration of a Bielliptic surface of type 2}\label{sec:degen_bielliptic}

 Semistable degenerations of surfaces of Kodaira dimension zero were classified by Kulikov \cite{Kulikov1977}, Persson--Pinkham \cite{persson1981degeneration}, and Morrison \cite{Morrison1981Semistable}. In the bielliptic case, among the possible central fibres are cycles of elliptic ruled surfaces (see \cite[Theorem 1.8 (1.8.2)]{Morrison1981Semistable}); chains of elliptic ruled surfaces occur only for types $1$ and $2$ (see \cite[Theorem 1.8 (1.8.3)(c)]{Morrison1981Semistable}).
\par
The latter are crucial for our purposes. We construct explicitly such a degeneration for a bielliptic surface of type $2$, thereby obtaining a situation to which the obstruction to the existence of a universal $0$-cycle introduced in \S \ref{sec:obstruction} applies.\par Theorem \ref{thm:intro_degeneration_S} follows from the more precise statement established below.

\begin{theorem}\label{thm:degeneration_S} Let $k$ be an algebraically closed field of characteristic different from $2,$ and let $E$ be an elliptic curve over $k$. Set $\Delta:=\Spec k[[t]],$ and fix an algebraic closure $K$ of its function field. Then there exists a regular, flat, projective scheme $\mathcal{S}\to\Delta$ together with a morphism $$\phi_{\mathcal{S}}\colon\mathcal{S}\longrightarrow E\times\Delta,$$ such that the following hold:
\begin{enumerate}
    \item\label{it:generic_fibre'} The geometric generic fibre $S_{K}$ is a bielliptic surface of type 2, and the induced morphism $\phi_{S_{K}}\colon S_{K}\to E\times_{k} K$ coincides with the Albanese fibration of $S_{K}$.
    \item\label{it:special_fibre'} The special fibre $S_{0}=R_{1}\cup R_{2}$ is a reduced simple normal crossings divisor on $\mathcal{S}$ with exactly two irreducible components. Each surface $R_{i}$ is a minimal ruled surface over an \'etale double cover $E_{i}\to E,$ and the two double covers are not isomorphic.\par
    The intersection $$C:=R_{1}\cap R_{2}$$ is a smooth, irreducible elliptic curve, embedded in each $R_{i}$ as a $2$-fold multisection. 
\end{enumerate}  
\end{theorem}
\begin{proof} Let $p_{\mathcal{F}}\colon\mathcal{F}\to\Delta$ be a minimal regular model of an elliptic curve with geometric generic fibre $F_{K}:=\mathcal{F}\times_{\Delta}K,$ and assume that the special fibre $F_{0}$ is of Kodaira type $\bf{I}_{4}$. Then $$F_{0}=C_{0}\cup C_{1}\cup C_{2}\cup C_{3},$$ where each $C_{i}\cong\bbP^{1}_{k}$. The components are arranged in a cycle: for $i\in\Z/4,$ $$C_{i}\cap C_{i+1}=\{\text{single node}\}, \quad C_{i}\cap C_{j}=\emptyset\ \text{if}\ j\not\equiv i\pm 1\ \text{(mod 4)}.$$
Each node $C_{i}\cap C_{i+1}$ is obtained by identifying the point $\infty\in C_{i}$ with the point $0\in C_{i+1}$.\par
The smooth locus $\mathcal{F}^{\mathrm{sm}}\to\Delta$ of $p_{\mathcal{F}}$ is the N\'eron model of its generic fibre. Let $\mathcal{F}^{\mathrm{sm}}[2]$ denote its $2$-torsion subgroup scheme; one has $$\mathcal{F}^{\mathrm{sm}}[2]\cong\mu_{2}\times(\Z/2)_{\Delta}.$$ This group scheme acts on $\mathcal{F},$ and the induced action on the special fibre is as follows: the factor $\mu_{2}$ acts on each component $C_{i}$ by multiplication by $\pm1,$ fixing the two nodes of $C_{i}$, while the factor $\Z/2$ permutes the components by sending $C_{i}$ to $C_{i+2}$.\par
Choose a non-trivial $2$-torsion class $$0\neq\epsilon\in\mathcal{F}^{\mathrm{sm}}[2]$$ whose restriction to the special fibre corresponds to the element $(1,\bar{1})\in\mu_{2}\times\Z/2,$ and let $\varphi_{\epsilon}\colon \ca F\longrightarrow\ca F$ denote translation by $\epsilon$.\par
The inversion morphism $$[-1]\colon F_{K}\longrightarrow F_{K},\quad x\longmapsto -x$$ extends canonically to an involution
\begin{equation*}
    \iota\colon \ca F\longrightarrow\ca F
\end{equation*}
over $\Delta$. On the special fibre $F_0,$ the automorphism $\iota$ induces, for each $i\in\Z/4,$ the isomorphism $$C_i\longrightarrow C_{-i},\quad (x_{0}:x_{1})\longmapsto (x_{1}:x_{0}).$$\par
Fix generators $a,b\in E[2],$ and let $\tau_{a},\tau_{b}\colon E\to E$ denote translation by $a$ and $b,$ respectively. Define
\begin{equation}\label{eq:inv_bielliptic}
    \sigma_{1}:=\tau_{a}\times_{k}\iota, \qquad 
    \sigma_{2}:=\tau_{b}\times_{k}\varphi_{\epsilon}.
\end{equation}
The involutions $\sigma_{i}$ act freely on $E\times_{k}\ca F,$ and they commute: $\sigma_{1}\circ\sigma_{2}=\sigma_{2}\circ\sigma_{1}$.\par
Set \begin{equation}\label{eq:deg_I}
\widetilde{\ca S}:=\bigl(E\times_{k}\ca F\bigr)/\sigma_{1}.\end{equation} The natural projection $\pr_{1}\colon E\times_{k}\ca F\to E\times_{k}\Delta$ descends to a morphism \begin{equation}\label{eq:albanese_I}\phi_{\widetilde{\ca S}}\colon \widetilde{\ca S}\longrightarrow \bigl(E/a\bigr)\times_{k}\Delta .\end{equation}
The family $\widetilde{\ca S}\to \Delta$ in \eqref{eq:deg_I} is a strictly semistable degeneration whose geometric generic fibre $\widetilde{S}_{K}$ is a bielliptic surface of type 1. The special fibre $\widetilde{S}_{0}$ has exactly three irreducible components $\widetilde{R}_{i},$ each a minimal elliptic ruled surface. More precisely, 
\begin{equation}\label{eq:components_I}
\widetilde{R}_{1} = \bigl(E \times C_0\bigr)/\sigma_1,\quad
\widetilde{R}_{2} =\bigl( E \times (C_1 \sqcup C_3)\bigr)/\sigma_1 \cong E \times \bbP^1_k,\quad
\widetilde{R}_{3} = \bigl(E \times C_2\bigr)/\sigma_1 .
\end{equation}

The dual graph of $\widetilde{S}_{0}$ is a chain with end components $\widetilde{R}_{1}$ and $ \widetilde{R}_{3}$. The intersections 
\begin{equation*}
    \widetilde{C}_{1,2}:= \widetilde{R}_{1}\cap \widetilde{R}_{2},\qquad \widetilde{C}_{2,3}:=\widetilde{R}_{2}\cap \widetilde{R}_{3}
\end{equation*}
are naturally isomorphic to the elliptic curve $E$. Moreover, the induced morphisms $$ \widetilde{C}_{1,2}\subset \widetilde{R}_{1} \longrightarrow E/a,\qquad \widetilde{C}_{2,3}\subset \widetilde{R}_{3} \longrightarrow E/a$$ coincide with the canonical double cover $$E\longrightarrow E/a.$$ In addition, the curves $\widetilde{C}_{i,i+1}$ embed as disjoint sections of the ruled surface $\widetilde{R}_{2}$.\par
The involution $\sigma_{2}$ in \eqref{eq:inv_bielliptic} descends to an automorphism $\widetilde{\sigma}_{2}$ of $\widetilde{\ca S}$. The quotient \begin{equation}\label{eq:deg_II}
    \ca S:=\widetilde{\ca S}/\widetilde{\sigma}_{2}\longrightarrow\Delta
\end{equation}
is a strictly semistable degeneration of a bielliptic surface of type 2. The involution $\widetilde{\sigma}_{2}$ identifies the components $\widetilde{R}_{1}$ and $\widetilde{R}_{3}$ and leaves $\widetilde{R}_{2}$ invariant. Under the isomorphism $\widetilde{R}_{2}\cong E \times \bbP^1_k$ in \eqref{eq:components_I}, its restriction is given by $$\sigma\colon\bigl(x,(x_{0}:x_{1})\bigr)\longmapsto \bigl(x+(a+b),(x_{1}:x_{0})\bigr).$$
It follows that the special fibre $S_{0}$ has exactly two irreducible components $R_{i},$ each a minimal elliptic ruled surface. More precisely, 
\begin{equation}\label{eq:components_II}
R_{1}\cong \widetilde{R}_{1}\cong \widetilde{R}_{3},\qquad R_{2}\cong \bigl(E\times\bbP^1_{k}\bigr)/\sigma.
\end{equation}
Their intersection $C:=R_{1}\cap R_{2}$ is isomorphic to $E,$ and embeds in each ruled surface $R_{i}$ as a double cover of its base $E_{i}$. In the notation of \eqref{it:special_fibre'}, the corresponding \'etale double covers $E_{i}\to E,\ (i=1,2),$ are given by \begin{equation}\label{eq:double_covers}
    E/a\longrightarrow E/E[2]\cong E,\qquad E/(a+b)\longrightarrow E/E[2]\cong E,
\end{equation}
respectively.\par
Finally, the morphism $\phi_{\widetilde{\ca S}}$ defined in \eqref{eq:albanese_I} descends, upon passage to the quotient by $\widetilde{\sigma}_{2}$, to a morphism 
\begin{equation}\label{eq:albanese_II}
\phi_{\ca S}\colon \ca S\longrightarrow \bigl(E/E[2]\bigr)\times_{k}\Delta\cong E\times_{k}\Delta,
\end{equation}
which restricts on the geometric generic fibre to the Albanese fibration of the bielliptic surface $S_{K}$.\par
This completes the proof.
\end{proof}

\begin{proof}[Proof of Theorem \ref{thm:intro_degeneration_S}] Let $\ca S\to\Delta$ be the degeneration constructed in Theorem \ref{thm:degeneration_S}, together with the induced morphism $$\phi_{\ca S}\colon \ca S\longrightarrow E\times_k\Delta.$$ Assertions \eqref{it:generic_fibre} and \eqref{it:special_fibre} of Theorem \ref{thm:intro_degeneration_S} follow directly from \eqref{it:generic_fibre'} and \eqref{it:special_fibre'} of Theorem \ref{thm:degeneration_S}. It therefore suffices to verify condition \eqref{it:albanese_S_0}.\par
To this end, we show that the sum morphism \begin{equation}\label{eq:sum}
    \Alb(R_{1})\times_k\Alb(R_{2})\longrightarrow E\end{equation}
admits no section under either assumption \ref{it:trivial_end} or \ref{it:CM_E}.\par
As in \eqref{eq:double_covers}, the natural morphisms $$ \Alb(R_{i})\longrightarrow E\qquad (i=1,2)$$ are given by the \'etale double covers 
$$E/a\longrightarrow E/E[2]\cong E,\qquad E/b\longrightarrow E/E[2]\cong E,$$
where $a,b\in E[2]$ are generators. Consequently, the morphism \eqref{eq:sum} can be identified with
$$E/a\times_k E/b\longrightarrow E/E[2]\cong E,$$
where the final isomorphism is induced by multiplication by $2$ on $E$.\par
Let $q_a\colon E\to E/a$ and $q_{b}\colon E\to E/b$ be the canonical quotient morphisms. Denote by $\hat{q}_a\colon E/a\to E$ and $\hat{q}_b\colon E/b\to E$ the dual isogenies of $q_a$ and $q_b,$ respectively. Assume first that property \ref{it:trivial_end} holds for $E$. Since $\End(E)=\Z,$ we have $$\Hom\bigl(E,E/a\bigr)=\Z\cdot q_a,\qquad \Hom\bigl(E,E/b\bigr)=\Z\cdot q_b .$$ It follows that any homomorphism $$f\colon E\longrightarrow E/a\times_{k} E/b$$ is of the form $f=([m_1]\circ q_a,[m_2]\circ q_b)$ for some integers $m_1,m_2$.\par
For any such $f$, its composition with the sum morphism \eqref{eq:sum} is trivial on $E[2].$ Hence the morphism \eqref{eq:sum} admits no section. This establishes condition \eqref{it:albanese_S_0} under the assumption of \ref{it:trivial_end}.\par
Next assume \ref{it:CM_E}. Let $$L:=\Q(\sqrt{c}), \qquad c\in\{-1,-2,-3,-11,-19,-43,-67,-163\},$$ and denote by $\mathcal O_{L}$ its ring of integers. By the Stark--Heegner theorem, $\mathcal{O}_L$ has class number one; hence $$\End(E)\cong\mathcal O_{L}$$ is a principal ideal domain.\par
For each of the above values of $c$, there exists a unique prime ideal $\mathfrak{p}\subset\mathcal{O}_{L}$ lying above $2$.\par Suppose, for the sake of contradiction, that there exist isogenies $$f_{a}\colon E\longrightarrow E/a,\qquad f_{b}\colon E\longrightarrow E/b$$
such that the composite morphism $$E\overset{(f_a,f_b)}{\longrightarrow}E/a\times E/b\overset{\hat{q}_a+\hat{q}_{b}}{\longrightarrow} E/E[2]\cong E$$
is the identity on $E$.\par
Define endomorphisms $$g_{a}:=\hat{q}_a\circ f_{a},\qquad  g_{b}:=\hat{q}_b\circ f_{b}.$$
Then $g_a$ and $g_b$ are isogenies of even degree, and by assumption $$g_a+g_b=\Id_E.$$ Since their degrees are divisible by $2,$ both lie in the unique prime ideal $\mathfrak{p}$ above $2$. However, the relation $g_a+g_b=1$ implies $$1\in(g_{a},g_{b}),$$ and therefore $1\in\mathfrak{p},$ a contradiction.\par
This proves condition \eqref{it:albanese_S_0} under the assumption of \ref{it:CM_E}, and completes the proof of Theorem \ref{thm:intro_degeneration_S}.
\end{proof}
\begin{example}\label{ex:example_with_section}
The Stark--Heegner theorem determines all quadratic imaginary number fields whose rings of integers are principal ideal domains. In particular, for $c<0$, the field $\Q(\sqrt{c})$ has class number one if and only if
$$
c\in\{-1,-2,-3,-7,-11,-19,-43,-67,-163\}.
$$

In the list of \ref{it:CM_E} (resp. \ref{it:CM_E_C}) we excluded the Heegner number $c=-7$. We now show that in this case the sum morphism
\begin{equation}\label{eq:q_a+q_b}
E/a \times E/b \overset{\hat{q}_a+\hat{q}_b}{\longrightarrow} E/E[2]\cong E
\end{equation}
admits a section for suitable generators $a,b\in E[2]$. Here $\hat{q}_a$ and $\hat{q}_b$ denote the dual isogenies of the canonical quotient morphisms
$$
q_a\colon E\longrightarrow E/a,
\qquad
q_b\colon E\longrightarrow E/b.
$$

Set
$$
\omega:=\frac{1+\sqrt{-7}}{2}.
$$
Then the ring of integers of $\Q(\sqrt{-7})$ is $\Z[\omega]$. Both $\omega$ and its complex conjugate $\bar{\omega}=1-\omega$ have norm $2$, and satisfy
$$
2=\omega\cdot\bar{\omega}.
$$
Let $g_\omega\in\End(E)\cong\Z[\omega]$ be the isogeny corresponding to $\omega$, which has degree $2$. Denote by $\hat{g}_\omega$ its dual isogeny, i.e. the endomorphism associated to $\bar{\omega}$. It follows that
$$
2\cdot\Id_E = g_\omega\circ\hat{g}_\omega,
\qquad
g_\omega+\hat{g}_\omega=\Id_E.
$$
The latter identity implies that the morphism \eqref{eq:q_a+q_b} admits a section for suitable generators $a,b\in E[2],$ as claimed.\par
Later, in Example \ref{ex:bielliptic_uni_0_cycle}, we show that elliptic curves with complex multiplication by the maximal order in $\Q(\sqrt{-7})$ give rise to explicit examples of bielliptic surfaces of type $2$ admitting a universal $0$-cycle. 
\end{example}

\begin{remark}\label{rem:degen_over_Q}
The construction of Theorem \ref{thm:degeneration_S} (equivalently, Theorem \ref{thm:intro_degeneration_S}) extends to yield a degeneration over $\Spec \Z_p^{\mathrm{sh}}$ whose geometric generic fibre is a bielliptic surface of type $2$ defined over $\bar{\bbQ}$ and whose special fibre satisfies \eqref{it:special_fibre'} and \eqref{it:albanese_S_0}.

Fix
$$
c \in \{-1,-2,-3,-11,-19,-43,-67,-163\},
$$
and let $E/\bbQ$ be an elliptic curve whose geometric endomorphism ring is the maximal order of the Heegner field $\Q(\sqrt{c})$; see \cite{lmfdb} for explicit examples. Choose an odd prime $p$ of good ordinary reduction for $E$. Let
$$
\pi \colon \ca E \longrightarrow \Spec \Z_{p}^{\mathrm{sh}}
$$
be the corresponding smooth model over the strict henselization of $\Z_p$; see \cite[Chapter I, \S4]{milne_etale}.

Let
$$
p_{\ca F} \colon \ca F \longrightarrow \Spec \Z_{p}^{\mathrm{sh}}
$$
be the minimal regular model of the elliptic curve $F/\bbQ$ defined by
$$
y^2 z = x(x-z)(x-p^4 z).
$$
Its special fibre is of Kodaira type $\mathbf{I}_4$. Replacing, in the proof of Theorem \ref{thm:degeneration_S}, the family
$$
E \times_k \ca F \longrightarrow \Delta
$$
by
$$
\ca E \times_{\Z_{p}^{\mathrm{sh}}} \ca F \longrightarrow \Spec \Z_{p}^{\mathrm{sh}},
$$
the argument carries over verbatim and yields the desired degeneration.
\end{remark}

\section{Main results}\label{sec:final}
\subsection{A bielliptic surface of type 2 admitting no universal 0-cycle} 
Theorem \ref{thm:main_result} follows immediately from the result below.
\begin{theorem}\label{thm:main_result_body}
Let $k$ be an algebraically closed field of characteristic $\neq 2.$ Let $E$ be an elliptic curve over $k$ such that either \ref{it:trivial_end} or \ref{it:CM_E} holds for $\End(E)$. Set $\Delta := \Spec k[[t]]$, and fix an algebraic closure $K$ of its function field. Let $\mathcal S \to \Delta$ be the degeneration constructed in Theorem \ref{thm:intro_degeneration_S}, together with the induced morphism
$$
\phi_{\mathcal S} \colon \mathcal S \longrightarrow E \times_k \Delta .
$$
Then for any algebraically closed field extension $F/K$, there exists no correspondence
$$
[\Gamma] \in \CH^{2}(E \times_k S_F)
$$
that induces a splitting of the pushforward map
\begin{equation*}
(\phi_{S_F})_{\ast} \colon \CH_0(S_F)_{\hom} \longrightarrow E_{F} .
\end{equation*}
\end{theorem}
\begin{proof} Suppose, for contradiction, that there exists an algebraically closed field extension $F/K$ and a correspondence
$$
[\Gamma] \in \CH^{2}(E \times_k S_F)
$$
such that the induced endomorphism $$(\phi_{S_F})_{\ast}\circ[\Gamma]_{\ast}\in\End\bigl(E_{F}\bigr)$$ is the identity.\par
Let $$S_0=\sum_{i=1}^nR_i$$ denote the special fibre of the family $\ca S\to\Delta,$ as in \eqref{it:special_fibre} of Theorem \ref{thm:intro_degeneration_S}. For each $i,$ let $\al_i\colon\Alb(R_i)\to E$ be the morphism of Albanese varieties induced by the restriction of $\phi_{\mathcal S}$ to $R_i$. By Corollary \ref{cor:obstruction}, the assumption above implies that the sum morphism 
$$\sum_{i=1}^n\al_i\colon\bigoplus_{i=1}^n\Alb(R_i)\longrightarrow E$$
admits a section. \par
This contradicts condition \eqref{it:albanese_S_0} of Theorem \ref{thm:intro_degeneration_S}. The resulting contradiction completes the proof.
\end{proof}

\begin{proof}[Proof of Theorem \ref{thm:main_result}] Let $E$ be an elliptic curve over $\bbC$ such that either \ref{it:trivial_end_C} or \ref{it:CM_E_C} holds for $\End(E)$. Choose a countable algebraically closed subfield $k_0\subset\bbC$ over which $E$ is defined; thus $$E\cong E_0\times_{k_{0}}\bbC$$ for some elliptic curve $E_{0}$ over $k_{0}$.\par
Let $$\ca S\longrightarrow\Delta:=\Spec k_{0}[[t]]$$ be the strictly semistable degeneration constructed in Theorem \ref{thm:intro_degeneration_S}, together with the induced morphism $$\phi_{\mathcal S} \colon \mathcal S \longrightarrow E_0 \times_{k_{0}} \Delta .$$\par
Since $k_0$ is countable, the inclusion $k_0\subset\bbC$ factors through the Laurent series field $k_0((t))$. Set $$S:=S_{k_{0}((t))}\times_{k_{0}((t))}\bbC.$$
By \eqref{it:generic_fibre} of Theorem \ref{thm:intro_degeneration_S}, the surface $S$ is a bielliptic surface of type 2, and the morphism $$\phi_{S}\colon S\longrightarrow E$$ induced by $\phi_{\ca S}$ coincides with the Albanese fibration of $S$.\par
In particular, the pushforward homomorphism $$(\phi_{S})_{\ast} \colon \CH_0(S)_{\hom} \longrightarrow E $$ agrees with the Abel--Jacobi map for $0$-cycles on $S$. It is well known that bielliptic surfaces have representable $\Ch$-group; see \cite[Example (2)]{bloch-srinivas}. Applying Theorem \ref{thm:main_result_body}, we conclude that there exists no correspondence $$
[\Gamma] \in \CH^{2}(E \times S)
$$
that induces a splitting of the Abel--Jacobi map 
\begin{equation*}
(\phi_{S})_{\ast} \colon \CH_0(S)_{\hom} \longrightarrow E.
\end{equation*}
Equivalently, the surface $S$ admits no universal $0$-cycle; cf. Definition \ref{def:universal_0-cycle}. This completes the proof of Theorem \ref{thm:main_result}.    
\end{proof}

\begin{example}\label{ex:bielliptic_uni_0_cycle}
We give explicit examples of bielliptic surfaces of type $2$ admitting a universal $0$-cycle.

Let $E$ be an elliptic curve over $\bbC$ with complex multiplication by the maximal order in the Heegner field $\Q(\sqrt{-7})$. A model defined over $\bbQ$ is given by the Weirstrass equation $$y^2z+xyz=x^3-x^2z-107xz^2+552z^3;$$ see \cite[\href{https://www.lmfdb.org/EllipticCurve/Q/49.a2}{Elliptic Curve 49.a2}]{lmfdb}. 

The ring of integers of $\Q(\sqrt{-7})$ is $\Z[\omega]$, where 
$$
\omega:=\frac{1+\sqrt{-7}}{2}.
$$
Both $\omega$ and its complex conjugate $\bar{\omega}=1-\omega$ have norm $2$, so that
$$
2=\omega\cdot\bar{\omega}.
$$

Let $g_{\omega}\in\End(E)\cong\Z[\omega]$ denote the endomorphism corresponding to $\omega$. It is an isogeny of degree $2$, and its dual isogeny $\hat g_{\omega}$ corresponds to $\bar{\omega}$. Hence
\begin{equation}\label{eq:deg_2_isog}
2\cdot\Id_{E}=g_{\omega}\circ\hat g_{\omega}, 
\qquad 
g_{\omega}+\hat g_{\omega}=\Id_E .
\end{equation}

The isogenies $g_{\omega}$ and $\hat g_{\omega}$ admit factorizations
$$
g_{\omega}\colon 
E \overset{f_{\omega}}{\longrightarrow} E/a 
\overset{\hat q_a}{\longrightarrow} E/E[2] \cong E,
\qquad
\hat g_{\omega}\colon 
E \overset{f_{\bar\omega}}{\longrightarrow} E/b 
\overset{\hat q_b}{\longrightarrow} E/E[2] \cong E,
$$
where $a,b\in E[2]$ are nontrivial $2$-torsion points. Here 
$$
q_a\colon E\longrightarrow E/a, 
\qquad 
q_b\colon E\longrightarrow E/b
$$
denote the canonical quotients, $\hat q_a$ and $\hat q_b$ are their dual isogenies, and the morphisms $f_{\omega}$ and $f_{\bar\omega}$ are isomorphisms.

Let $F$ be an elliptic curve over $\bbC$. Consider the action of the Klein four-group $\Gi=(\Z/2)^{\oplus 2}$ on $E\times F$ generated by the involutions
$$
\sigma_1(x,y)=(x+a,-y), 
\qquad 
\sigma_2(x,y)=(x+a+b,y+c),
$$
where $0\neq c\in F[2]$. The quotient
$$
S:=\bigl(E\times F\bigr)/\Gi
$$
is a bielliptic surface of type $2$.

The surface $S$ admits two natural elliptic fibrations
$$
\phi_S\colon S\longrightarrow E/E[2]\cong E,
\qquad
\psi_S\colon S\longrightarrow F/\Gi\cong \bbP^1_{\bbC}.
$$
The fibration $\phi_S$ is smooth and coincides with the Albanese morphism of $S$, while $\psi_S$ has exactly four multiple fibres of multiplicity $2$; see \cite[Table 2]{Serrano1990}. Two of these fibres have reduced structure isomorphic to $E/a$, and the remaining two have reduced structure $E/b$.

Fix embeddings
$$
E/a\longhookrightarrow S,
\qquad
E/b\longhookrightarrow S,
$$
and consider the subvarieties
$$
\Gamma_{f_\omega}\subset E\times (E/a)\subset E\times S,
\qquad
\Gamma_{f_{\bar\omega}}\subset E\times (E/b)\subset E\times S,
$$
given by the graphs of $f_\omega$ and $f_{\bar\omega}$, respectively. Set
$$
[\Gamma]:=[\Gamma_{f_\omega}]+[\Gamma_{f_{\bar\omega}}]\in\CH^2(E\times S).
$$

We claim that $[\Gamma]$ defines a universal $0$-cycle on $S$. Indeed, the pushforward
$$
(\phi_S)_*\colon \CH_0(S)_{\hom}\longrightarrow E
$$
coincides with the Abel--Jacobi map for $0$-cycles on $S$. By construction the induced endomorphisms
$$
(\phi_S)_*\circ[\Gamma_{f_\omega}]_*,
\qquad
(\phi_S)_*\circ[\Gamma_{f_{\bar\omega}}]_*
$$
identify with the isogenies $g_\omega$ and $\hat g_\omega$ of $E$. Therefore the second identity in \eqref{eq:deg_2_isog} implies
$$
(\phi_S)_*\circ[\Gamma]_\ast=\Id_E.
$$
Thus $S$ admits a universal $0$-cycle.
\end{example}

\subsection{A counterexample to the integral Hodge conjecture}\label{subsec:int_Hodge} We now apply Theorem \ref{thm:main_result} to prove Corollary \ref{cor:integral_Hodge_S}. Specifically, we construct a threefold $X$ of Kodaira dimension zero for which the integral Hodge conjecture fails for $1$-cycles. The variety $X$ is defined as the product of a bielliptic surface of type $2$, which admits no universal $0$-cycle, with its Albanese variety, an elliptic curve $E$. Thus $X$ is an elliptic fibration over the abelian surface $E \times E$. In contrast to the counterexamples of Benoist and Ottem \cite{benoist-ottem}, the non-algebraic Hodge class in degree $4$ produced by our construction is non-torsion.\par
Let $X$ be a smooth projective variety over $\bbC$ of dimension $d= \dim X$. We begin by recalling how the non-existence of a universal $0$-cycle on $X$ is related to the failure of the integral Hodge conjecture. Throughout, all (co)homology groups are taken to be Betti (co)homology, and the subscript "tf" denotes the quotient by torsion.\par

The following lemma is well-known.
\begin{lemma}\label{lem:natural_identification} There is a natural isomorphism 
\begin{equation}\label{eq:natural_identification}
    \Hom\bigl(H_{1}\bigl(\Alb(X),\Z\bigr),\,H_{1}\bigl(X,\Z\bigr)_{\mathrm{tf}}\bigr)\cong H^{1}\bigl(\Alb(X),\Z\bigr)\otimes_{\Z} H^{2d-1}\bigl(X,\Z\bigr)_{\mathrm{tf}}.
\end{equation}
\end{lemma}

\begin{proof} For any lattice $L$, we write
$$
L^\vee := \Hom(L, \mathbb{Z})
$$
for its dual. By Poincar\'e duality for the smooth projective variety $X$, there is a natural
isomorphism
$$
H_{1}\bigl(X,\Z\bigr)_{\mathrm{tf}} \;\cong\; H^{2d-1}\bigl(X,\Z\bigr)_{\mathrm{tf}} .
$$
The universal coefficient
theorem yields a canonical identification
$$
H^{2d-1}\bigl(X,\Z\bigr)_{\mathrm{tf}}
\;\cong\;
H_{2d-1}\bigl(X,\Z\bigr)^\vee.
$$
Combining these, we obtain
$$
\Hom\bigl(H_{1}\bigl(\Alb(X),\Z\bigr),\, H_{1}\bigl(X,\Z\bigr)_{\mathrm{tf}}\bigr)
\;\cong\;
\Hom\bigl(H_{1}\bigl(\Alb(X),\Z\bigr),\, H_{2d-1}\bigl(X,\Z\bigr)^\vee\bigr).
$$

By the tensor--Hom adjunction, this group is naturally isomorphic to
$$
\bigl(H_{1}\bigl(\Alb(X),\Z\bigr)\otimes_{\Z} H_{2d-1}\bigl(X,\Z\bigr)\bigr)^\vee.
$$
Since both $H_{1}(\Alb(X),\Z)$ and $H_{2d-1}(X,\Z)$ are finitely generated
free abelian groups, there is a canonical identification
$$
\bigl(H_{1}\bigl(\Alb(X),\Z\bigr)\otimes_{\Z} H_{2d-1}\bigl(X,\Z\bigr)\bigr)^\vee
\;\cong\;
H_{1}\bigl(\Alb(X),\Z\bigr)^\vee
\otimes_{\Z}
H_{2d-1}\bigl(X,\Z\bigr)^\vee.
$$

Finally, another application of the universal coefficient theorem identifies
$$
H_{1}\bigl(\Alb(X),\Z\bigr)^\vee\cong H^{1}\bigl(\Alb(X),\Z\bigr),
\quad
H_{2d-1}\bigl(X,\Z\bigr)^\vee\cong H^{2d-1}\bigl(X,\Z\bigr)_{\mathrm{tf}},
$$
which completes the proof.
\end{proof}

By construction, the Albanese morphism of $X$
\begin{equation}\label{eq:al_X}
    \al_{X}\colon X\longrightarrow\Alb(X)
\end{equation}
induces an isomorphism of lattices 
\begin{equation}\label{eq:iso_lattices}
    (\al_{X})_{\ast}\colon H_{1}\bigl(X,\Z\bigr)_{\tf}\overset{\sim}{\longrightarrow}H_{1}\bigl(\Alb(X),\Z\bigr).
\end{equation}
\par
Via the natural identification \eqref{eq:natural_identification} of Lemma \ref{lem:natural_identification}, the inverse of $(\al_{X})_{\ast}$ in \eqref{eq:iso_lattices} defines a cohomology class 
\begin{equation}\label{eq:Hodge_class}\delta_{X}\in H^{1}\bigl(\Alb(X),\Z\bigr)\otimes_{\Z} H^{2d-1}\bigl(X,\Z\bigr)_{\tf}\subset H^{2d}\bigl(\Alb(X)\times X,\Z\bigr)_{\text{tf}},\end{equation}
where the inclusion is induced by the Künneth decomposition.\par

By \cite[Corollary 1.6.3]{Murre}, the inverse of $(\al_{X})_{\ast}\otimes_{\Z}\Q$ is induced by a correspondence in $$\CH^{d}(\Alb(X)\times X)\otimes_{\Z}\Q.$$ It then follows from a theorem of Lieberman (see \cite[Theorem 2A11.]{Kleiman1968}) that the class $$\delta_{X}\in H^{2d}(\Alb(X)\times X,\Q)$$ is rationally algebraic.\par

We record the following lemma, which addresses the integral version of this question; see \cite[\S 2]{voisin_a}.

\begin{lemma}\label{lem:int_Hodge_conj_vs_uni_0_cycl}Let $X$ be a smooth projective variety over $\bbC$ of dimension $d = \dim X,$ and let $$\cl^d\colon\CH^{d}(\Alb(X)\times X)\longrightarrow H^{2d}(\Alb(X)\times X,\Z)_{\mathrm{tf}}$$ be the cycle class map. The class $$\delta_{X}\in H^{2d}\bigl(\Alb(X)\times X,\Z\bigr)_{\mathrm{tf}}$$ constructed in \eqref{eq:Hodge_class} is a Hodge class of degree $2d.$ Moreover, there exists a cycle $$[\Gamma]\in\CH^{d}(\Alb(X)\times X)$$ whose Künneth component of type $(1,2d-1)$ in $\cl^{d}([\Gamma])$ equals $\delta_X$ if and only if $X$ admits a universal $0$-cycle. In particular, if $X$ does not admit a universal $0$-cycle, then $\delta_{X}$ is not algebraic.
\end{lemma}
We are now in a position to prove Corollary \ref{cor:integral_Hodge_S}.

\begin{proof}[Proof of Corollary \ref{cor:integral_Hodge_S}]Let $E$ be an elliptic curve over $\bbC$ such that either \ref{it:trivial_end_C} or \ref{it:CM_E_C} holds for $\End(E)$. 
By Theorem \ref{thm:main_result}, there exists a bielliptic surface $S$ of type 2 with $\Alb(S)\cong E$ such that $S$ admits no universal $0$-cycle; cf. Definition \ref{def:universal_0-cycle}.
By Lemma \ref{lem:int_Hodge_conj_vs_uni_0_cycl}, it follows that the cohomology class $$\delta_{S}\in H^{4}\bigl(E\times S,\Z\bigr)_{\mathrm{tf}}$$ constructed in \eqref{eq:Hodge_class} is a Hodge class which is not algebraic.\par
This shows that the integral Hodge conjecture in degree $4$ fails for $E\times S$, even modulo torsion.\end{proof}

\begin{example}\label{ex:counterexample_Q}
We can construct examples of surfaces as in Theorem \ref{thm:main_result} that are defined over $\bar{\bbQ}$. As a consequence, this also yields threefolds as in Corollary \ref{cor:integral_Hodge_S} which are defined over number fields. 
The argument follows the idea of Hasset and Tschinkel, namely to use
specialization to characteristic $p>0$; see \cite[Remarque 5.9]{CTVoisin2012} and \cite{totaro2013integral}.

Let $E$ be an elliptic curve over $\bbQ$ such that its geometric endomorphism ring $\End(E_{\bar{\bbQ}})$ satisfies \ref{it:CM_E}; explicit examples can be found in \cite{lmfdb}. 
Combining the construction of Remark \ref{rem:degen_over_Q} with the specialization argument of Corollary \ref{cor:obstruction'}, we obtain a bielliptic surface $S$ of type $2$ over $\bar{\bbQ}$ whose Albanese variety is isomorphic to $E_{\bar{\bbQ}}$. 

Moreover, $S$ has the following property: for every algebraically closed field extension $F/\bar{\bbQ}$, there exists no correspondence
$$
[\Gamma] \in \CH^{2}(E_F \times S_F)
$$
that induces a splitting of the Abel--Jacobi map
$$
\alpha_S \colon \Ch(S_F)_{\hom} \longrightarrow \Alb(S_F) \cong E_F.
$$
Equivalently, $S$ does not admit a universal $0$-cycle over any field extension $F/\bar{\bbQ}$; cf. Definition \ref{def:universal_0-cycle}. 
\end{example}
\begin{remark}\label{rem:no_counterexample_F_p}
Assume the Tate conjecture for divisors on smooth projective surfaces over finite fields.
Let $S$ be a smooth projective surface over a finite field $\bbF_q$ of characteristic $p$, whose Albanese variety is an elliptic curve $E$, and let
\[
\phi_S\colon S\longrightarrow E
\]
be its Albanese fibration. Under this assumption, $S$ admits a universal $0$-cycle over $\overline{\bbF}_q$ whenever $p$ does not divide $\text{index}(\phi_S)$.\par
Indeed, a theorem of Schoen \cite[Theorem (0.5)]{Schoen1998}, conditional on the above Tate conjecture, implies that for every prime $\ell\neq p$ the cycle class map
\[
\cl^{2}\colon
\CH^{2}\bigl((E\times S)_{\overline{\bbF}_q}\bigr)
\otimes_{\Z}\Z_{\ell}
\longrightarrow
\varinjlim_{k/\bbF_q}
H^{4}\bigl((E\times S)_{\overline{\bbF}_q},\Z_{\ell}(2)\bigr)^{
\Gal(\overline{\bbF}_q/k)}
\]
is surjective, where the limit runs over all finite extensions $k/\bbF_q$.\par
For the remainder of the argument, all varieties and morphisms are considered over $\overline{\bbF}_q$.
Let $C$ denote the cokernel of the natural homomorphism
\[
\CH^2(E\times S)\longrightarrow\End(E),\qquad
[\Gamma]\longmapsto\psi_{(E,[\Gamma])},
\]
where $\psi_{(E,[\Gamma])}$ is the endomorphism defined in \eqref{eq:psi}.
Arguing as in \S\ref{subsec:int_Hodge}
(cf.\ Lemma \ref{lem:int_Hodge_conj_vs_uni_0_cycl}), we obtain
\[
C\otimes_{\Z}\Z\!\left[\tfrac{1}{p}\right]=0.
\]
Thus the class of $\Id_E$ in $C$ is annihilated by a power of $p$.
On the other hand,
\[
\text{index}(\phi_S)\cdot\Id_E=0
\quad\text{in }C;
\]
cf.\ \cite[Proposition 2.7]{voisin_a}.
If $p$ is coprime to $\text{index}(\phi_S)$, these two annihilation relations imply that $\Id_E=0$ in $C$, and hence that $S$ admits a universal $0$-cycle.\par
Consequently, assuming the Tate conjecture for divisors on smooth projective surfaces over finite fields, every bielliptic surface of type $2$ over $\overline{\bbF}_p$ with $p>2$ admits a universal $0$-cycle; see Table \ref{tab:bielliptic_coh}.
\end{remark}

\section*{Acknowledgments} The author is grateful to Olivier Wittenberg and Matteo Tamiozzo for helpful comments and discussions during his visit to Universit\'e Sorbonne Paris Nord. The author also thanks Olivier de Gaay Fortman and the anonymous referee for their comments that helped to improve the exposition.
\printbibliography
\end{document}